\documentclass[12pt,oneside]{amsart}
\usepackage{amscd,amsmath,amssymb,amsfonts,a4}
\usepackage{pstricks}
\usepackage{color}
\usepackage{latexsym}
\usepackage{eepic}
\usepackage{epsfig}
\usepackage{graphicx}
\usepackage{eufrak}
\usepackage{textcomp}
\usepackage{mathrsfs}
\usepackage[english]{babel}

\DeclareMathAlphabet{\mathpzc}{OT1}{pzc}{m}{it}
\newtheorem{theorem}{Theorem}[section]
\newtheorem{theorem-definition}[theorem]{Theorem-Definition}
\newtheorem{lemma-definition}[theorem]{Lemma-Definition}
\newtheorem{definition-prop}[theorem]{Proposition-Definition}

\newtheorem{prop}[theorem]{Proposition}
\newtheorem{lemma}[theorem]{Lemma}
\newtheorem{cor}[theorem]{Corollary}

\newtheorem{definition}[theorem]{Definition}

\theoremstyle{definition}
\newtheorem{question}[theorem]{Question}
\newtheorem{remark}[theorem]{Remark}
\newtheorem{example}[theorem]{Example}

\newcommand{\LL}{\ensuremath{\mathbb{L}}}

\newcommand{\Z}{\ensuremath{\mathbb{Z}}}
\newcommand{\Q}{\ensuremath{\mathbb{Q}}}

\newcommand{\R}{\ensuremath{\mathbb{R}}}
\newcommand{\C}{\ensuremath{\mathbb{C}}}

\newcommand{\Pro}{\ensuremath{\mathbb{P}}}
\newcommand{\A}{\ensuremath{\mathbb{A}}}

\renewcommand{\R}{\ensuremath{\mathbb{R}}}
\renewcommand{\C}{\ensuremath{\mathbb{C}}}
\renewcommand{\Pro}{\ensuremath{\mathbb{P}}}
\renewcommand{\A}{\ensuremath{\mathbb{A}}}

\newcommand{\Spec}{\ensuremath{\mathrm{Spec}\,}}
\newcommand{\Spf}{\ensuremath{\mathrm{Spf}\,}}

\newcommand{\redu}{\mathrm{red}}
\newcommand{\rig}{\mathrm{rig}}
\newcommand{\an}{\mathrm{an}}
\newcommand{\sing}{\mathrm{sing}}
\newcommand{\Aut}{\mathrm{Aut}}
\newcommand{\Aff}{\mathrm{Aff}}

\newcommand{\Sch}{\mathrm{Sch}}
\newcommand{\Sets}{\mathrm{Sets}}
\newcommand{\Var}{\mathrm{Var}}
\newcommand{\Sm}{\mathrm{Sm}}
\newcommand{\GSch}{G\textendash \Sch}
\newcommand{\KVar}[1]{\ensuremath{K_0(\Var_{#1})}}
\newcommand{\Kmod}[1]{\ensuremath{K_0^{\mathrm{mod}}(\Var_{#1})}}
\newcommand{\KG}[1]{\ensuremath{K_0(G\textendash\Var_{#1})}}
\newcommand{\KR}[1]{\ensuremath{K^R_0(\Var_{#1})}}

\numberwithin{equation}{section} \hyphenpenalty=6000
\begin{document}
\title[Finite group actions and weak N\'eron models]{Finite group actions,
rational fixed points and weak N\'eron models}
\author{H\'el\`ene Esnault}
\address{
Universit\"at Duisburg-Essen, Mathematik, 45117 Essen, Germany}
\email{esnault@uni-due.de}
\author{Johannes Nicaise}
\address{ Katholieke Universiteit Leuven, Departement Wiskunde, Celestijnenlaan
200B,
B-3001 Leuven, Belgium } \email{Johannes.Nicaise@wis.kuleuven.be}
\dedicatory{\`A la m\'emoire d'Eckart Viehweg, en souvenir de la lumi\`ere dont
il nous inondait}
\thanks{The first author was partially supported by  the DFG Leibniz Preis, the SFB/TR45, the ERC
Advanced Grant 226257. The second author was partially supported
by the Fund for Scientific Research - Flanders (project
G.0415.10).}
\begin{abstract} If $G$ is a finite $\ell$-group acting on an affine space
$\A^n_K$ over a finite  field $K$ of cardinality prime to $\ell$,
Serre \cite{serre-finite} shows that there exists a rational fixed
point. We generalize this to the case where $K$ is a henselian
discretely valued field of characteristic zero with algebraically
closed residue field and with residue characteristic different
from $\ell$. We also treat the case where the residue field is
finite of cardinality $q$ such that $\ell$ divides $q-1$. To this
aim, we study group actions on weak N\'eron models.
\end{abstract} \maketitle

\section{Introduction}
Let $\ell$ be a prime number, and let $G$ be a finite $\ell$-group
acting on an affine space $\A_K^n$ over a field $K$ of
characteristic different from $ \ell$ (here we view $\A^n_K$ as a
$K$-variety, not as the $K$-group $\mathbb{G}^n_{a,K}$). Serre
observes in \cite{serre-finite} that the closed subscheme $
(\A^n_K)^G$ of fixed points is not empty. By standard arguments,
it is enough to show this if $K$ is a finite field. Over a finite
field, non-trivial orbits of the set of rational points have
cardinality divisible by $\ell$, thus there must exist fixed
points that are rational. Serre raises the question
\cite{serre-finite} whether rational fixed points always exist,
that is, whether
 $ (\A^n_K)^G(K)$ is empty or not.

The aim of this note is to discuss Serre's question for a field
$K$ whose Galois group is close to $\widehat{\mathbb{Z}}$, the
Galois group of a finite field. The main theorem (\ref{thm-rat})
and its corollary (\ref{cor-finite}) assert that if $K$  is a
henselian discretely valued field of characteristic $0$, with
algebraically closed residue field and residue characteristic
different from $\ell$, then Serre's question has a positive
answer. Furthermore, if $K$ is henselian with finite residue field
of cardinality $q$ such that $\ell$ divides $q-1$, then Serre's
question has a positive answer as well.

While Serre's counting argument sketched above over a finite field
$K$ is very elementary, one can also argue cohomologically. Serre
shows that, for any base field $K$, the Gysin  map in compactly
supported \'etale cohomology
\begin{equation}\label{eq-gysin}H_c^{*-2c}((\A^n_K)^G \times_K
K^a, \mathbb{F}_\ell(-c))\xrightarrow{\gamma_{K^a}}
H_c^*(\A^n_{K^a}, \mathbb{F}_\ell)\end{equation} is an
isomorphism, where $K^a$ is an algebraic closure of $K$ and $c$ is
the codimension of $(\A^n_K)^G$ in $\A^n_K$. Now assume that $K$
is finite, of cardinality $q$. Since the Gysin map is
Galois-equivariant and $\A^n_K$ has trivial cohomology, it follows
from the isomorphism \eqref{eq-gysin} that the trace of
 the geometric Frobenius on the graded cohomology space
$$\bigoplus_{i=0}^{2n-2c} H_c^{i}((\A^n_K)^G \times_K K^a, \mathbb{F}_\ell)$$
 is equal to $q^{n-c}$.
The Grothendieck-Lefschetz trace formula allows one to conclude
the existence of a rational fixed point.

The second author has shown that a similar trace formula exists
when the base field is a strictly henselian discretely valued
field $K$ of equicharacteristic zero (there are partial results in
arbitrary characteristic \cite{nicaise-tame}, but there the
situation is much more subtle due to issues of wild ramification).
He showed in \cite{nicaise} that, for any $K$-variety $X$, the
trace of any monodromy operator (i.e., any topological generator
of the absolute Galois group ${\rm Gal}(K^a/K)$) on the compactly
supported $\ell$-adic cohomology of $X$ is equal to the {\em
rational volume} $s(X)$ of $X$. The rational volume is a certain
measure for the set of rational points of $X$ (see Definition
\ref{def-ratl-vol}). If $X$ admits a weak N\'eron model
$\mathcal{X}$ over the valuation ring of $K$, then $s(X)$ equals
the $\ell$-adic Euler characteristic of the special fiber of
$\mathcal{X}$. The invariant $s(X)$ vanishes if $X(K)=0$, so that
it can be used to detect the existence of a rational point. The
rational volume is a specialization of Loeser and Sebag's {\em
motivic Serre invariant} \cite{motrigid} and its generalization in
\cite{nicaise}; see Section \ref{sec-motserre}.

The existence of such a trace formula is surprising, because one
cannot characterize rational points with intersection theory over
 $K^a$, since the action of the
monodromy operator cannot be represented by an algebraic cycle.
The proof of the trace formula uses resolution of singularities
and explicit computations on nearby cycles. It is a challenging
problem to find a more conceptual proof that does not use
resolution of singularities.

When applied to Serre's question, the trace formula and the
isomorphism (\ref{eq-gysin}) imply that the rational volume of
$(\A^n_K)^G$ is equal to one, so that $(\A^n_K)^G(K)$ is
non-empty. We also consider the case where $K$ has mixed
characteristic. Here we cannot use the trace formula due to issues
of wild ramification. Rather than using the whole information
contained in the rational volume, we consider it only modulo
$\ell$. Under the assumption that the residue field of $K$ has
characteristic different from $\ell$, we show that
$s(\mathbb{A}^n_K)$ and $s((\mathbb{A}^n_K)^G)$ are congruent
modulo $\ell$. Intuitively, this can still be considered as a
manifestation of the isomorphism (\ref{eq-gysin}) on the level of
the rational volume. Since $s(\A^n_K)=1$, we find that
$(\mathbb{A}^n_K)^G(K)$ is non-empty.


The geometric part of the work consists in showing the
compatibility  of weak N\'eron models with group actions.

\subsection*{Acknowledegements:}  This work was started during a
summer school at the Feza G\"ursey Institute in Istanbul in June
2010. We thank the Turkish mathematicians for their hospitality.
 We thank Nguy\^e\~{n}  Duy T\^an and Jean-Pierre Serre for
careful reading and interesting questions. The second author is
grateful to Raf Cluckers for inspiring discussions.  We thank
 the referee for a friendly precise reading and for valuable suggestions. In
particular, Lemma \ref{lemm-finitenorm} is due to him.


\section{Preliminaries and notations}\label{sec-notation}

\subsection{Notations} \label{subsection-notations}
We denote by $R$ a discrete valuation ring with quotient field $K$
and perfect residue field $k$. We denote by $p$ the characteristic
exponent of $k$ (thus $p=1$ if $k$ has characteristic zero, and
$p$ equals the characteristic of $k$ if it is non-zero) and by
$\frak{m}$ the maximal ideal of $R$. We fix a strict henselization
$R^{sh}$ of $R$, and we denote by $K^{sh}$ its quotient field. The
residue field $k^s$ of $R^{sh}$ is a separable closure of $k$. We
denote by $v_{K}$ the discrete valuation on $K$. We choose a value
$\varepsilon$ in $]0,1[$ and we define an absolute value
$|\cdot|_{K}$ on $K$ by $|0|_{K}=0$ and
$$|x|_{K}=\varepsilon^{v_K(x)}$$ for $x\in K^{\times}$.
If $k$ is finite, of cardinality $q$, we take
$\varepsilon=q^{-1}$. The absolute value $|\cdot|_K$ extends
canonically to $K^{sh}$.

For any $R$-scheme $\mathcal{X}$, we put
\begin{eqnarray*}
\mathcal{X}_k&=&\mathcal{X}\times_R k
\\ \mathcal{X}_K&=&\mathcal{X}\times_R K.
\end{eqnarray*}

For every scheme $S$, we denote by $S_{\redu}$ the maximal reduced
closed subscheme of $S$. An $S$-{\it variety} is a reduced
separated $S$-scheme of finite presentation.

If $F$ is a field with separable closure $F^s$,  $X$ is a
separated $F$-scheme of finite type, and $\ell$ is a prime
invertible in $F$,  we define the {\it Euler characteristic with
proper supports} of $X$ by
$$\chi_c(X)=\sum_{i\geq 0}(-1)^i \mathrm{dim}\, H^i_c(X\times_F
F^s,\Q_\ell).$$
 This Euler characteristic is independent of $\ell$. If $F$ has characteristic
zero,
  this follows from comparison with singular cohomology; if $F$ is finite, from
the
  cohomological interpretation of the zeta function ($\chi_c(X)$ is equal to
minus the degree of the Hasse-Weil zeta function of $X$); the general case is
deduced from the finite field case
  by a spreading out argument and proper base change.

For any set $\mathscr{S}$, we denote its cardinality by
$|\mathscr{S}|$.
\subsection{Group actions}
Let $G$ be a finite group. We say that an action of $G$ on a
scheme $S$ is {\em good} if every orbit is contained in an affine
open subscheme of $S$. This condition is automatically fulfilled,
for instance, if $S$ is quasi-projective over an affine scheme, by
\cite[3.3.36(b)]{liu}.

Let $S$ be a scheme, with trivial $G$-action. We denote by
$(\Sch/S)$ the category of $S$-schemes, and by $(\GSch/S)$ the
category of $S$-schemes with $G$-action. By
\cite[Propositon~3.1]{edixhoven},
the functor
$$(\Sch/S)\to (\GSch/S)$$ that endows an $S$-scheme with the trivial $G$-action
has a right adjoint, which we denote by
$$(\cdot)^G:(\GSch/S)\to (\Sch/S):X\mapsto X^G.$$
The $S$-scheme $X^G$ is called the {\em fixed locus} of the
$G$-action on $X$. By definition, it represents the functor
$$(\Sch/S)\to (\Sets):T\mapsto X(T)^G.$$ The functor $(\cdot)^G$
commutes with arbitrary base change $S'\to S$. If $X$ is separated
over $S$, then the tautological morphism
$$X^G\to X$$ is a closed immersion \cite[ loc.cit.]{edixhoven}.

\begin{prop}[
also
\cite{edixhoven}, Proposition 3.4]\label{prop-smooth} If $X\to S$
is smooth, and $|G|$ is invertible on $X$, then $X^G\to S$ is
smooth.
\end{prop}
\begin{remark}
Proposition \ref{prop-smooth} was shown in \cite[Prop.\,1.3]{I}
 when $S$ is the spectrum of an algebraically closed field.
\end{remark}

\subsection{Grothendieck rings}
Let $F$ be a field.  We denote by $\KVar{F}$ the {\it Grothendieck
ring of varieties over} $F$. As an abelian group, $\KVar{F}$ is
defined by the following presentation:
\begin{itemize}
\item {\em generators:} isomorphism classes $[X]$ of separated
$F$-schemes of finite type $X$,
 \item {\em relations:} if $X$ is
a separated $F$-scheme of finite type and $Y$ is a closed
subscheme of $X$, then
$$[X]=[Y]+[X\setminus Y].$$
These relations are called {\em scissor relations}.
\end{itemize}
By the scissor relations, one has $[X]=[X_{\redu}]$  for every
separated $F$-scheme of finite type $X$, and $[\emptyset]=0$. We
endow the group $\KVar{F}$ with the unique ring structure such
that
$$[X]\cdot [X']=[X\times_F X']$$ for all $F$-varieties $X$
and $X'$. The identity element for the multiplication is the class
$[\Spec F]$ of the point. For a detailed survey on the
Grothendieck ring of varieties, we refer to \cite{NiSe-K0}.

We denote by
$$\Kmod{F}$$ the {\em modified Grothendieck ring of
varieties over} $F$ \cite[\S\,3.8]{NiSe-K0}. This is the quotient
of $\KVar{F}$ by the ideal $\mathcal{I}_F$ generated by elements
of the form
$$[X]-[Y]$$
where $X$ and $Y$ are separated $F$-schemes of finite type such
that there exists a finite, surjective, purely inseparable
$F$-morphism
$$Y\to X.$$
For instance, if $F$ has characteristic $p>0$, $X\to S$ is a
morphism of $F$-varieties and $X^{(p)}$ is the pullback of $X$
over the absolute Frobenius morphism on $S$, then $[X^{(p)}]=[X]$
in $\Kmod{F}$, since the relative Frobenius
$$\mathrm{Frob}_{X/S}:X\to X^{(p)}$$ is a finite, surjective,
purely inseparable $F$-morphism.

If $F$ has characteristic zero, then it is easily seen that
$\mathcal{I}_F$ is the zero ideal \cite[3.11]{NiSe-K0}. It is not
known if $\mathcal{I}_F$ is non-zero if $F$ has positive
characteristic. In particular, if $F'$ is a non-trivial finite
purely inseparable extension of $F$, it is not known whether
$[\Spec F']\neq 1$ in $\KVar{F}$.

Let $G$ be a finite group.  We denote by $\KG{F}$ the {\it equivariant
Grothendieck ring of $F$-varieties with good $G$-action}. It is the
quotient of the free abelian group generated by isomorphism
classes $[X]$ of separated $F$-schemes of finite type $X$ with
good $G$-action, by the subgroup generated by elements of the form
$[X]-[Y]-[X\setminus Y]$ with $X$ a separated $F$-scheme of finite
type with good $G$-action and $Y$ a $G$-stable closed subscheme of
$X$. The product on $\KG{F}$ is defined by
$$[X]\cdot [X']=[X\times_F X']$$ where $G$ acts diagonally on
$X\times_F X'$.

 With a slight abuse of notation, we denote by the same symbol $\LL_F$
the class of $\A^1_F$ -- with trivial $G$-action, if applicable --
in all of the above Grothendieck rings.

If $R$ has equal characteristic, then we set
$$\KR{k}=\KVar{k}.$$

If $R$ has mixed characteristic, then
we set $$\KR{k}=\Kmod{k}.$$

There exists a unique ring morphism
$$\chi_c:\KR{k}/(\LL_k-1)\to \Z$$ that sends the class of a separated
$k$-scheme of finite type $X$ to the Euler characteristic with
proper supports $\chi_c(X)$ defined in section
\ref{subsection-notations} (see \cite[4.2 and 4.14]{NiSe-K0}).
Moreover, if $k$ is finite, of cardinality $q$, there exists a
unique ring morphism
$$\sharp:\KR{k}/(\LL_k-1)\to \Z/(q-1)$$ that sends the class of a separated
$k$-scheme of finite type $X$ to the residue class modulo $(q-1)$
of  $|X(k)|$ (the cardinality of the set of $k$-rational points on
$X$) \cite[ loc.cit.]{NiSe-K0}.

\section{The Serre invariant}\label{sec-motserre}
\begin{definition}
We say that a $K$-variety $X$ is {\em bounded} if the set
$X(K^{sh})$ is bounded in $X$ in the sense of \cite[1.1.2]{neron}.
\end{definition}
Every proper $K$-variety is bounded, by \cite[1.1.6]{neron}. If
$n$ is a strictly positive integer, then a closed subvariety $X$
of
$$\A^n_K=\Spec K[x_1,\ldots,x_n]$$ is bounded if and only if, for
every $i\in \{1,\ldots,n\}$, the function
$$X(K^{sh})\to \R^+:a\mapsto |x_i(a)|_K$$ is bounded. In particular, $\A^n_K$ is
not bounded.

\begin{prop}[\cite{neron}, Corollary 3.5.7]\label{prop-bounded}
A $K$-variety $X$ is bounded if and only if there exists an
$R$-variety $\mathcal{X}$ endowed with a $K$-isomorphism
$$\mathcal{X}_K\to X$$ such that the natural map
$$\mathcal{X}(R^{sh})\to X(K^{sh})$$ is a bijection.
\end{prop}

\begin{definition}\label{def-weakner}
Let $X$ be a smooth $K$-variety. A {\em weak N\'eron model} for
$X$ is a smooth $R$-variety $\mathcal{X}$ endowed with an
isomorphism
$$\mathcal{X}_K\to X$$ such that the natural map
$$\mathcal{X}(R^{sh})\to X(K^{sh})$$ is a bijection.
\end{definition}
  Note that $X(K^{sh})$ is
empty if and only if the special fiber $\mathcal{X}_k$ is empty,
by \cite[2.3.5]{neron} (applied to $\mathcal{X}\times_R R^{sh}$).
Thus if $X(K^{sh})$ is empty, then up to isomorphism, $X$ is the
unique weak N\'eron model of $X$.

\begin{theorem}\label{thm-weakner}
If $X$ is a smooth $K$-variety, then $X$ admits a weak N\'eron
model if and only if  $X$ is bounded.
\end{theorem}
\begin{proof}
If $X$ admits a weak N\'eron model, then $X$ is bounded by
Proposition \ref{prop-bounded}. Assume, conversely, that $X$ is
bounded, and let $\mathcal{X}$ be an $R$-model of $X$ as in
Proposition \ref{prop-bounded}. By \cite[3.4.2]{neron}, there
exists a projective morphism $\mathcal{X}'\to \mathcal{X}$ such
that the induced morphism $\mathcal{X}'_K\to \mathcal{X}_K$ is an
isomorphism and every element of $\mathcal{X}'(R^{sh})$ factors
through the $R$-smooth locus $\Sm(\mathcal{X}')$ of
$\mathcal{X}'$. By the valuative criterion of properness, the map
$\mathcal{X}'(R^{sh})\to \mathcal{X}(R^{sh})$ is a bijection. It
follows that $\Sm(\mathcal{X}')$ is a weak N\'eron model of $X$.
\end{proof}

\begin{definition}\label{def-smoothbound}
Let $X$ be a $K$-variety. We say that $X$ is of type (N) if the
$K$-smooth locus $\Sm(X)$ of $X$ contains all $K^{sh}$-points of
$X$ and $\Sm(X)$ is bounded.
\end{definition}
 In particular, every smooth
and proper $K$-variety is of type (N). If $X$ is of type (N), then
$\Sm(X)$ admits a weak N\'eron model.

\begin{theorem}\label{thm-LoSe}
Let $X$ be a $K$-variety of type (N), and let $\mathcal{X}$ be a
weak N\'eron model for $\Sm(X)$. The class of $\mathcal{X}_k$ in
the ring
$$\KR{k}/(\LL_k-1)$$ only depends on $X$, and not on the
choice of a weak N\'eron model $\mathcal{X}$.
\end{theorem}
\begin{proof}
This result was originally proven by Loeser and Sebag for smooth
quasi-compact rigid $K$-varieties and their formal weak N\'eron
models, under the assumption that $K$ is complete; see
\cite[4.5.3]{motrigid}. It was extended to smooth and bounded
rigid $K$-varieties in \cite[5.11]{NiSe-weilres}. A rigid
$K$-variety $X$ is called bounded if it admits a quasi-compact
open subset $U$ such that for every finite unramified extension
$K'$ of $K$, $U$ contains all $K'$-points of $X$.

Theorem \ref{thm-LoSe} is easily deduced from this case, as
follows. Replacing $X$ by $\Sm(X)$, we may assume that $X$ is
smooth. Denote by $\widehat{R}$ and $\widehat{K}$ the completions
of $R$ and $K$, respectively. If $X$ is a smooth $K$-variety of
type (N), then the rigid analytification $(X\times_K
\widehat{K})^{\rig}$ is smooth and bounded, by
\cite[5.2.1]{conrad} and \cite[4.3]{nicaise}. Moreover, if
$\mathcal{X}$ is a weak N\'eron model for $X$, then its formal
$\frak{m}$-adic completion $\widehat{\mathcal{X}}$ over
$\widehat{R}$ is a formal weak N\'eron model for $(X\times_K
\widehat{K})^{\rig}$, by \cite[4.9]{nicaise}, and the special
fibers  $\mathcal{X}_k$ and $\widehat{\mathcal{X}}\times_{\Spf
\widehat{R}}\Spec k$ are canonically isomorphic.
\end{proof}

\begin{definition}\label{def-LoSe}
Let $X$ be a  $K$-variety of type (N), and let $\mathcal{X}$ be a
weak N\'eron model for $\Sm(X)$. The {\em motivic Serre invariant}
of $X$, denoted by $S(X)$, is the class of $\mathcal{X}_k$ in the
ring
$$\KR{k}/(\LL_k-1).$$
\end{definition}
This definition only depends on $X$, by Theorem \ref{thm-LoSe}.
Note that, if $X$ is a $K$-variety such that $X(K^{sh})$ is empty,
then $X$ is of type (N) and $S(X)=0$. More generally, we can view
$S(X)$ as a measure for the set of $K^{sh}$-rational points on
$X$: one can think of $X(K^{sh})$ as a family of open unit balls
parameterized by $\mathcal{X}_k$. If $K$ is complete, this
intuitive picture can be made more precise using the language of
rigid geometry \cite[2.5]{NiSe-survey}.

\begin{remark}
In \cite{motrigid} and \cite{nicaise}, the motivic Serre invariant
was defined with values in $$\KVar{k}/(\LL_k-1).$$ In the
meantime, Julien Sebag and the second author discovered a flaw in
the change of variables formula for motivic integrals in mixed
characteristic, upon which the proof of Theorem \ref{thm-LoSe} was
based. To correct it, one has to replace $\KVar{k}$ by $\KR{k}$.
This is explained in detail in \cite{NiSe-note} and
\cite{NiSe-survey}. The correction is harmless for all
applications, since all realization morphisms of $\KVar{k}$ that
are used in practice factor through $\Kmod{k}$
\cite[4.13]{NiSe-K0}. In fact, if $k$ has positive characteristic,
it is not even known if the projection
$$\KVar{k}\to \Kmod{k}$$ is an isomorphism.
\end{remark}

The notion of motivic Serre invariant was first introduced by
Loeser and Sebag in \cite[\S\,4.5]{motrigid}. It was inspired by
Serre's invariant that classifies compact manifolds over a
non-archimedean local field \cite{serre}. Assume that $K$ is
complete and $k$ finite, of cardinality $q$. Let $M$ be a compact
$K$-analytic manifold of pure dimension $d$ (in the na\"ive sense,
not a rigid variety; see \cite{cartan}  and
\cite[\S\,2.4]{igusa}). Serre showed that
 $M$ is isomorphic to a disjoint union
of closed unit balls $R^d$, and that the number $N$ of balls is
well-defined modulo $(q-1)$. We call the class of $N$ in
$\Z/(q-1)$ the Serre invariant of $M$, denoted by $S_{\sharp}(M)$.

Let $X$ be a  $K$-variety of type (N). The space $X(K)$, with its
$K$-adic topology, carries a canonical structure of compact
$K$-analytic manifold, which we denote by $X^{\an}$. The following
Proposition compares the Serre invariant of $X^{\an}$ to the
motivic Serre invariant of $X$.

\begin{prop}
\label{prop-pserre} Assume that $K$ is complete and $k$ finite, of
cardinality $q$. Let $X$ be a  $K$-variety of type (N), of pure
dimension $d$. Then the Serre invariant $S_{\sharp}(X^{\an})$ of
$X^{\an}$ is the image of $S(X)$ under the point counting
realization
$$\sharp:\KR{k}/(\LL_k-1)\to \Z/(q-1):[X]\mapsto |X(k)|.$$
\end{prop}
\begin{proof}
This was shown in \cite[4.6.3]{motrigid} in the case where $X$ is
a smooth quasi-compact rigid $K$-variety. Proposition
\ref{prop-pserre} is easily deduced from that case, as follows. We
may assume that $X$ is smooth. Let $\mathcal{X}$ be a weak N\'eron
model for $X$, and let $\widehat{\mathcal{X}}$ be its formal
$\frak{m}$-adic completion. The generic fiber
$\widehat{\mathcal{X}}_\eta$ is a smooth quasi-compact open rigid
subvariety of the rigid analytification $X^{\rig}$ of $X$ that
contains all $K$-points of $X^{\rig}$, so the $K$-analytic
manifold associated to $X$ is canonically isomorphic to the
$K$-analytic manifold associated to $\widehat{\mathcal{X}}_\eta$.
Moreover, the motivic Serre invariants of
$\widehat{\mathcal{X}}_\eta$ and $X$ are the same: since
$\widehat{\mathcal{X}}$ is a formal weak N\'eron model of
$\widehat{\mathcal{X}}_\eta$, both motivic Serre invariants are
equal to the class of $\mathcal{X}_k$ in $\KR{k}/(\LL-1)$.

Let us briefly sketch the proof of \cite[4.6.3]{motrigid},
translated to the set-up of Proposition \ref{prop-pserre}. We may
assume that $X$ is smooth. Let $\mathcal{X}$ be a weak N\'eron
model for $X$. One has a specialization map
$$sp_{\mathcal{X}}:X(K)=\mathcal{X}(R)\to \mathcal{X}(k)$$
defined by reduction modulo the maximal ideal of $R$. For every
point $x$ of $\mathcal{X}(k)$, the fiber
$sp_{\mathcal{X}}^{-1}(x)$ is an open subset of $X(K)$, and it
inherits from $X^{\an}$ a structure of $K$-analytic manifold.
Smoothness of $\mathcal{X}$ over $R$ implies that
$sp_{\mathcal{X}}^{-1}(x)$ is isomorphic to the closed unit disc
$R^d$. As a $K$-analytic manifold, $X^{\an}$ is isomorphic to the
disjoint union
$$\bigsqcup_{x\in \mathcal{X}(k)}sp_{\mathcal{X}}^{-1}(x)$$ so that
the $p$-adic Serre invariant of $X^{\an}$ equals the class of
$|\mathcal{X}(k)|$ in $\Z/(q-1)$. This is precisely the image of
$$S(X)=[\mathcal{X}_k]$$ under the point counting
realization
$$\sharp:\KR{k}/(\LL_k-1)\to \Z/(q-1).$$
\end{proof}

\begin{lemma}\label{lemm-complete}
Let $X$ be a $K$-variety, and denote by $\widehat{K}$ the
completion of $K$. Then $X$ is of type (N) if and only if
$X\times_K \widehat{K}$ is of type (N). Moreover,
$$S(X)=S(X\times_K
\widehat{K}).$$
\end{lemma}
\begin{proof}
By the canonical isomorphism
$$\Sm(X\times_K\widehat{K})\cong \Sm(X)\times_K \widehat{K}$$ we may
assume that $X$ is smooth. Then $X$ is bounded if and only if
$X\times_K \widehat{K}$ is bounded, by \cite[4.4]{nicaise}. Denote
by $\widehat{R}$ the completion of $R$. If $\mathcal{X}$ is a weak
N\'eron model of $X$, then $\mathcal{X}\times_R \widehat{R}$ is a
weak N\'eron model of $X\times_K \widehat{K}$, by
\cite[3.6.7]{neron}, and the special fibers of these weak N\'eron
models are canonically isomorphic.
\end{proof}

If $K$ has characteristic zero, the second author extended the
construction of the motivic Serre invariant to arbitrary
$K$-varieties, in the following way.

\begin{theorem}\label{thm-char0}
Assume that $K$ has characteristic zero. There exists a unique
ring homomorphism
$$S:\KVar{K}\to \KR{k}/(\LL_k-1)$$ that sends $[Y]$ to
$S(Y)$ for every smooth and proper $K$-variety $Y$. The morphism
$S$ sends $\LL_K$ to $1$.

Let $X$ be a $K$-variety and $U$ be a subvariety of $X$. If $U$
 contains all the points in $X(K^{sh})$, $U$ is smooth over $K$ and
 $U$  admits a weak N\'eron model $\mathcal{U}$, then $S([X])$ is
 the class of $\mathcal{U}_k$ in $\KR{k}/(\LL_k-1)$. In particular,
 if $X$ has type (N), then $S([X])=S(X)$.
\end{theorem}
\begin{proof}
If $K$ is complete, this is \cite[5.4]{nicaise}. The proof is
based on Weak Factorization (it uses Bittner's presentation of the
Grothendieck ring \cite{bittner}) and on a generalization of
N\'eron smoothening to pairs of $K$-varieties to control the
behaviour of the motivic Serre invariant under blow-up.

The general case can easily be deduced as follows. We denote by
$\widehat{K}$ the completion of $K$ and by
$$S_{\widehat{K}}:\KVar{\widehat{K}}\to \KR{k}/(\LL_k-1)$$
 the ring homomorphism
from \cite[5.4]{nicaise}.  Uniqueness of $S$ follows from the fact
that $\KVar{K}$ is generated by the classes of smooth and
projective $K$-varieties, by the scissor relations and Hironaka's
resolution of singularities. To prove existence, we define $S$ to
be the composition of the base change morphism
$$\KVar{K}\to \KVar{\widehat{K}}:[X]\mapsto [X\times_K
\widehat{K}]$$ with
$$S_{\widehat{K}}:\KVar{\widehat{K}}\to \KR{k}/(\LL_k-1).$$
It follows from Lemma \ref{lemm-complete} that $S([Y])=S(Y)$ for
every $K$-variety $Y$ of type (N). In particular, if $Y(K^{sh})$
is empty, then $S([Y])=0$. If $X$ and $U$ are as in the statement,
then we can partition $X\setminus U$ into subvarieties of $X$
without $K^{sh}$-points. Additivity of $S$ implies that
$S([X])=S([U])$, and we know that $S([U])=S(U)$ because $U$ is of
type (N). By definition, $S(U)$ is the class of $\mathcal{U}_k$ in
$\KR{k}/(\LL_k-1)$.
\end{proof}

\begin{definition}[\cite{nicaise}, Definition 5.5]\label{def-char0}
Assume that $K$ has characteristic zero. For any $K$-variety $X$,
we define the {\em motivic Serre invariant} of $X$ by
$$S(X)=S([X])\quad \in \KR{k}/(\LL_k-1).$$
\end{definition}
If $X$ is of type (N), this definition is compatible with
Definition \ref{def-LoSe}, by Theorem \ref{thm-char0}.

\begin{definition} \label{def-ratl-vol}
Let $X$ be a $K$-variety. Assume either that $K$ has
characteristic zero, or that $X$ is of type (N). We define the
{\em rational volume} of $X$, denoted by $s(X)$, as
$$s(X)=\chi_c(S(X))\quad \in \Z.$$
\end{definition}
\begin{remark}
If $X$ has type (N) and $\mathcal{X}$ is a weak N\'eron model of
$\Sm(X)$, then by definition, $s(X)=\chi_c(\mathcal{X}_k)$. To our
knowledge, there is no general proof of the independence of $s(X)$
of the choice of $\mathcal{X}$ that does not use the change of
variables formula for motivic integrals. If the residue field $k$
has characteristic zero, this independence property can also be
deduced from the trace formula in Theorem \ref{thm-trace}.
\end{remark}

The rational volume $s(X)$ vanishes if $X(K^{sh})$ is empty. So
one way to detect the existence of a $K^{sh}$-rational point on
$X$ is to show the non-vanishing of the rational volume of $X$. If
$R$ is henselian and the residue field $k$ is finite, the
following proposition provides a similar technique to detect the
existence of a $K$-rational point.

\begin{prop}\label{prop-serrefinite}
Suppose that $R$ is henselian and that $k$ is finite, of
cardinality $q$. Let $X$ be a $K$-variety.  Assume either that $K$
has characteristic zero, or that $X$ is of type (N). If
$$\sharp S(X) \neq 0$$ in $\Z/(q-1)$, then $X(K)$ is non-empty.
\end{prop}
\begin{proof}
If $X$ is of type (N), then $\Sm(X)$ admits a weak N\'eron model
$\mathcal{X}$ over $R$, and $$\sharp S(X)=|\mathcal{X}_k(k)| \mod
(q-1).$$ Since $R$ is henselian and $\mathcal{X}$ is smooth over
$R$, any element of $\mathcal{X}_k(k)$ lifts to an element of
$\mathcal{X}(R)=X(K)$, so that $X(K)$ is non-empty if $\sharp
S(X)\neq 0$.

It remains to prove the case where $K$ has characteristic zero.
Here we can proceed by induction on the dimension of $X$. If $X$
has dimension zero, then $X$ is of type (N) and we are done.
Assume that $\mathrm{dim}(X)>0$, and that the result holds for all
$K$-varieties of strictly smaller dimension. Since $K$ has
characteristic zero and $X$ is reduced, the $K$-smooth locus
$\Sm(X)$ is open and dense in $X$, and $Y=X\setminus \Sm(X)$ has
strictly smaller dimension than $X$. Since $$\sharp
S(\Sm(X))+\sharp S(Y)=\sharp S(X) \neq 0,$$ we know that $\sharp
S(Y) \neq 0$ or $\sharp S(\Sm(X)) \neq 0$. If $\sharp S(Y) \neq 0$
then $Y$ has a $K$-rational point by the induction hypothesis, so
that $X$ also has a $K$-rational point. Thus, we may assume that
$\sharp S(\Sm(X)) \neq 0$, and it suffices to consider the case
where $X$ is smooth over $K$.

Let $\overline{X}$ be a smooth compactification of $X$, and put
$Z=\overline{X}\setminus X$. Since $$\sharp S(\overline{X})-\sharp
S(Z)=\sharp S(X) \neq 0,$$ we know that $\sharp S(Z) \neq 0$ or
$\sharp S(\overline{X}) \neq 0$. If $\sharp S(Z) \neq 0$ then $Z$,
and thus $\overline{X}$, has a $K$-rational point by the induction
hypothesis, since $Z$ has strictly smaller dimension than $X$. If
$\sharp S(\overline{X})\neq 0$, then $\overline{X}$
 also has a $K$-rational point because $\overline{X}$ is smooth and
proper over $K$, so that $\overline{X}$ is of type (N). Thus
$\overline{X}(K)$ is Zariski-dense in $\overline{X}$, because $K$
is a henselian valued field and therefore a large field by
\cite[3.1]{pop} (applied to the set of localities
$\mathcal{L}=\{K\}$). It follows that $X(K)$ is non-empty.
  \end{proof}
\begin{remark}
Large fields are also called {\em ample} or {\em fertile} in the
literature. Every henselian valued field is large. In
\cite[3.1]{pop}, this fact is deduced from the implicit function
theorem for henselian valued fields.
\end{remark}

\section{$G$-models}\label{sec-G}
Let $G$ be a finite group.

\begin{definition}
Let $X$ be a $K$-variety endowed with a good action of $G$. A
$G$-model for $X$ is a flat $R$-variety $\mathcal{X}$ that carries
a good action of $G$, endowed with a $G$-equivariant isomorphism
$$\mathcal{X}_K\to X.$$
\end{definition}

The following lemma was kindly suggested to us by the referee.
\begin{lemma}\label{lemm-finitenorm}
Let $\mathcal{Z}$ be an $R$-variety. Assume either that $R$ is
excellent, or that the generic fiber $\mathcal{Z}_K$ is
geometrically reduced. Then the normalization morphism
$\widetilde{\mathcal{Z}}\to \mathcal{Z}$ is finite.
\end{lemma}
\begin{proof}
If $R$ is excellent, this follows from \cite[33.H]{matsumura}, so
it is enough to consider the case where $\mathcal{Z}_K$ is
geometrically reduced.  We denote by $\widehat{R}$ and
$\widehat{K}$ the completions of $R$ and $K$, respectively. We may
assume that $\mathcal{Z}$ is affine, integral and $R$-flat, and we
set $A=\mathcal{O}(\mathcal{Z})$. We denote by $\widetilde{A}$ the
normalization of $A$, and we set $B=A\otimes_R \widehat{R}$. Then
$B$ is $\widehat{R}$-flat and thus a subring of $A\otimes_R
\widehat{K}$. Since $\mathcal{Z}_K$ is geometrically reduced, $B$
is reduced.

If $C$ is an $A$-algebra such that every element of $C$ is
integral over $A$, then $C$ is finite over $A$ if and only if it
is finitely generated. Thus if $\widetilde{A}$ were not finite
over $A$, then by joining elements from $\widetilde{A}$ to $A$, we
could construct a strictly ascending chain
$$A\varsubsetneq A_1\varsubsetneq A_2 \varsubsetneq \ldots$$
of finite sub-$A$-algebras
 of the fraction field
$\mathrm{Frac}(A)$ of $A$. Tensoring with the faithfully flat
$R$-algebra $\widehat{R}$, we would obtain a
 strictly ascending chain
\begin{equation}\label{eq-chain}B=A\otimes_R \widehat{R}\varsubsetneq A_1\otimes_R
\widehat{R}\varsubsetneq A_2\otimes_R \widehat{R} \varsubsetneq
\ldots\end{equation} of finite sub-$B$-algebras of
$\mathrm{Frac}(A)\otimes_K \widehat{K}$.
 The ring
$\mathrm{Frac}(A)\otimes_K \widehat{K}$ is a subring of the total
ring of fractions of $B$. It follows that the elements of the
chain \eqref{eq-chain} are all contained in the normalization of
$B$, which is finite over $B$ by excellence of $\widehat{R}$. This
contradicts the fact that \eqref{eq-chain} is strictly ascending.
 Thus $\widetilde{A}$ is finite over $A$.
\end{proof}
\begin{prop}\label{prop-Gmod}
Let $X$ be a normal proper $K$-variety, endowed with a good action
of $G$. Assume either that $R$ is excellent, or that $X$ is
geometrically reduced. Then $X$ admits a proper $G$-model
$\mathcal{X}$ over $R$.
\end{prop}
\begin{proof}
Since the $G$-action on $X$ is good, the quotient $X/G$ is
representable by a $K$-variety, and the projection morphism
$$X\to X/G$$ is finite and surjective
\cite[V.1.5 and V.1.8]{sga1}. Since $X$ is proper over $K$, we
know that $X/G$ is proper over $K$ \cite[5.4.3]{ega2}. Since
taking $G$-quotients commutes with extension of the base field $K$
\cite[V.1.9]{sga1}, we know that $X/G$ is geometrically reduced if
$X$ is geometrically reduced.

By Nagata's embedding theorem, $X/G$ admits a flat and proper
model $\mathcal{Y}$ over $R$. Let $\mathcal{X}$ be the
normalization of $\mathcal{Y}$ in $X$. Then $\mathcal{X}$ is
normal, and $\mathcal{X}_K$ is canonically isomorphic to $X$. The
normalization map $\mathcal{X}\to \mathcal{Y}$ is finite, by Lemma
\ref{lemm-finitenorm}.

 Thus $\mathcal{X}$ is
a flat proper model for $X$ over $R$. The action of $G$ on $X$
induces an action on the total field of functions $K(X)$ of $X$,
and, since $G$ acts trivially on $\mathcal Y$, the $G$-action on
$X$ extends uniquely to an action on $\mathcal{X}$. The $G$-action
on $\mathcal{X}$ is good since $\mathcal{X}$ is finite over
$\mathcal{Y}$ and $G$ acts trivially on $\mathcal{Y}$.
\end{proof}

\begin{definition}
Let $X$ be a smooth $K$-variety, endowed with a good $G$-action. A
weak N\'eron $G$-model of $X$ is an $R$-smooth $G$-model
$\mathcal{X}$ for $X$ such that the natural map
$$\mathcal{X}(R^{sh})\to X(K^{sh})$$ is bijective.
\end{definition}

\begin{prop}\label{prop-weakG}
If $X$ is a smooth and proper $K$-variety, endowed with a good
$G$-action, 
 then $X$ admits a
weak N\'eron $G$-model.
\end{prop}
\begin{proof}
By Proposition \ref{prop-Gmod}, $X$ admits a  proper $G$-model
$\mathcal{X}$ over $R$. The
 smoothening algorithm in the proof of \cite[3.4.2]{neron}
 produces a canonical sequence of blow-ups
 $$\mathcal{X}^{(r)}\to \ldots \mathcal{X}^{(1)}\to
 \mathcal{X}^{(0)}=\mathcal{X}$$
 such that, for every $i$ in $\{0,\ldots,r-1\}$, the center $C_i$ of the blow-up
 $$\mathcal{X}^{(i+1)}\to \mathcal{X}^{(i)}$$ is a reduced closed
 subscheme of the special fiber $\mathcal{X}^{(i)}_k$ and such
 that $\Sm(\mathcal{X}^{(r)})$ is a weak N\'eron model for $X$.
  Inspecting the proof of \cite[3.4.2]{neron}, we see that the center $C_i$ is
 stable under any $R$-automorphism of $\mathcal{X}^{(i)}$. It
 follows that all the models $\mathcal{X}^{(i)}$ of $X$ carry a
 unique extension of the $G$-action on $X$ and that all the blow-up morphisms
are $G$-equivariant.
 Moreover, the $G$-action on
 $\Sm(\mathcal{X}^{(r)})$ is good, since the $G$-action on
 $\mathcal{X}$ is good and the blow-up morphisms are projective.
  We conclude that $\Sm(\mathcal{X}^{(r)})$ is a weak N\'eron $G$-model
 of $X$.
\end{proof}

\begin{prop}\label{prop-Gfix}
Let $X$ be a smooth $K$-variety with good $G$-action, and let
$\mathcal{X}$ be a weak N\'eron $G$-model for $X$. Assume
 that the order $|G|$ of $G$ is prime to $p$. Then the $G$-fixed
locus $\mathcal{X}^G$ is a weak N\'eron model for $X^G$.
\end{prop}
\begin{proof}
The scheme $\mathcal{X}^G$ is smooth over $R$, by Proposition
\ref{prop-smooth}. 
 Moreover, we have
$$\mathcal{X}^G(R^{sh})=
\big(\mathcal{X}(R^{sh})\big)^G=\big(X(K^{sh})\big)^G=X^G(K^{sh}).
$$ This shows  that $\mathcal{X}^G$ is a weak N\'eron model for
$X^G$.
\end{proof}


\section{Serre's question}\label{sec-serre}
In \cite[\S\,1]{serre-finite}, Serre raises the following
question.

\begin{question}\label{qu-serre}
Let $F$ be a field, and let $G$ be a finite $\ell$-group, for some
prime $\ell$ invertible in $F$. Suppose that $G$ acts on $\A^n_F$
for some integer $n>0$. Does this action always admit an
$F$-rational fixed point?
\end{question}

\begin{remark}\label{rem-serre}
As Serre points out, even the case $F=\Q$, $n=3$ and $|G|=2$ is
open \cite[p.2]{serre-finite}.
\end{remark}

It is clear that the condition that $\ell$ is invertible in $F$
cannot be omitted. For instance, for every field $F$ of
characteristic $\ell>1$, the additive group
$\mathbb{G}_a(\mathbb{F}_\ell)$ acts on $\A^1_F$ by translation,
and this action does not admit a fixed point.

We now discuss some cases where Serre's question can be answered.
We start by recalling
 in Sections \ref{subsec-finite} and \ref{subsec-algclosed}
some cases that are due to Serre. In Section \ref{subsec-lowdim},
we show that Serre's question has a positive answer in dimension
$\leq 2$, as implied by Serre in Remark \ref{rem-serre}. We also
discuss the methods Serre uses, which immediately yield the real
case (see Section \ref{sub-r}). Finally, we treat in Section
\ref{sub-hens} the case where the ground field is a henselian
discretely valued field of characteristic zero, with algebraically
closed residue field of characteristic different from $\ell$. We
show that, in that case as well, Serre's question has a positive
answer.

{\it For the remainder of  Section \ref{sec-serre}, we maintain
the assumption that $G$ is a finite $\ell$-group, with $\ell$ a
prime invertible in the field $F$.}
\subsection{The case where $F$ is finite}\label{subsec-finite}
\begin{prop}[see \cite{serre-finite}, proof of Theorem 1.2.] \label{prop-finfield}
Let $F$ be a finite field, and let $X$ be an $F$-variety with
$G$-action. Then
$$|X(F)|\equiv |X^G(F)|\mod \ell.$$
In particular, if $X=\A^n_F$ for some integer $n>0$, then $X^G(F)$
is non-empty.
\end{prop}
\begin{proof}
For every orbit $O$ of the $G$-action on $X(F)$, the cardinality
of $O$ is a power of $\ell$. Thus, either $O$ consists of a single
fixed point, or the cardinality of $O$ is divisible by $\ell$.
Since the orbits form a partition of $X(F)$, we find
$$|X(F)|\equiv |X^G(F)|\mod \ell.$$
If $X=\A^n_F$ then $|X(F)|=|F|^n$, and this value is not divisible
by $\ell$. It follows that $X^G(F)$ is non-empty.
\end{proof}
\subsection{The case where $F$ is separably
closed}\label{subsec-algclosed} Serre gives several proofs for
this case. In fact, he makes the stronger assumption that $F$ is
algebraically closed, but this is not necessary: if $X$ is a
smooth $F$-variety with $G$-action, then $X^G$ is smooth
(Proposition \ref{prop-smooth}). Hence, if $F$ is separably closed, then
$X^G$ has an $F$-rational point as soon as it is non-empty.

The proof of Theorem 1.2 in \cite{serre-finite} uses a spreading
out argument to reduce to the case where $F$ is finite. Four other
proofs make use of \'etale cohomology with $\ell$-adic
coefficients  \cite[\S\,7]{serre-finite}. Let us briefly recall
three of them, since these arguments will reappear in another
context below.

\subsubsection{Multiplicativity of the Euler characteristic in tame
coverings} \label{subsubsec-euler}
Serre proves the following
result.
\begin{prop}[Serre \cite{serre-finite}, Section 7.2]\label{prop-euler}
Let $F$ be any field, and let $X$ be an $F$-variety with good
$G$-action. Then one has  the following congruence of $\ell$-adic
Euler characteristics:
$$\chi_c(X)\equiv \chi_c(X^G)\mod \ell.$$
In particular, if $\ell$ does not divide $\chi_c(X)$, then $X^G$
is non-empty.
\end{prop}
\begin{proof}
\textit{Case 1: $|G|=\ell$.} Then $G$ acts freely on the open
subvariety $Y=X\setminus X^G$ of $X$, so that the projection $Y\to
Y/G$ is a finite \'etale covering of degree $\ell$. Since $\ell$
is different from the characteristic of $F$, one has
$$\chi_c(Y)=\ell\cdot \chi_c(Y/G)$$ (see e.g. \cite{illusie}).
By additivity of the Euler characteristic, one
has
$$\chi_c(X)=\chi_c(X^G)+\ell\cdot \chi_c(Y/G)$$ and the result follows.

\textit{Case 2: general case.} We use an induction on the
cardinality of $G$.   Assume that $|G|>\ell$, and that we know the
result for  all finite $\ell$-groups $H$ with $|H|<|G|$. Since all
finite $\ell$-groups are solvable,
 the group $G$ admits a
non-trivial normal subgroup $H$.
 Since
$$X^{G}=(X^H)^{G/H}$$ one has
$$\chi_c(X^G)\equiv \chi_c(X^H)\equiv \chi_c(X)\mod \ell.$$
\end{proof}
\begin{cor}\label{cor-euler}
Assume that $F$ is separably closed, and let $X$ be a smooth
$F$-variety with good $G$-action. If $\ell$ does not divide
$\chi_c(X)$, then $X^G(F)$ is non-empty. In particular, if
$X=\A^n_F$ for some integer $n>0$, then $X^G(F)$ is non-empty.
\end{cor}

\subsubsection{The Grothendieck-Lefschetz-Verdier trace formula} \label{verdier}
Assume that $F$ is separably closed. Let $X$ be an irreducible
$F$-variety of dimension $n$ with $G$-action, and assume that
$H^i_c(X,\Q_\ell)=0$ for $i\neq 2n$. This condition is fulfilled,
for instance, if $X$ is $\A^n_F$.
 Since $H^{2n}_c(X,\Q_\ell)$ has dimension one, and $G$ acts
trivially on this cohomology space, we find that $t(g)=1$ for
every element $g$ of $G$, where
$$t(g)=\sum_{i\geq 0}(-1)^i
\mathrm{Trace}(g\,|\,H^i_c(X,\Q_\ell)).$$ On the other hand, the
Grothendieck-Lefschetz-Verdier trace formula \cite[III]{sga5}
implies that $t(g)=0$ if $g$ does not have a fixed point on $X$.
It follows that the $G$-action on $X$ admits a fixed point if $G$
is cyclic. If $X$ is smooth, there exists even an $F$-rational
fixed point, by smoothness of $X^G$.

\subsubsection{Smith theory} This approach yields the most detailed
information about
 the fixed locus of $G$. In \cite[7.5]{serre-finite}, Serre
proves the following result, using Smith theory (again, Serre
makes the stronger assumption that $F$ is algebraically closed,
but the case where $F$ is separably closed follows from this by
invariance of \'etale cohomology under purely inseparable base
change).

\begin{theorem}\label{thm-smith}
Assume that $F$ is separably closed. Let $X$ be an $F$-variety
with good $G$-action. Assume that $H^i(X,\Z/\ell)=0$ for $i>0$
and that $X$ is connected and
 non-empty. Then $H^i(X^G,\Z/\ell)=0$ for $i>0$ and
$X^G$ is connected and non-empty.
\end{theorem}
\begin{cor}\label{cor-smith}
Assume that $F$ is separably closed, and let $n$ be a non-zero natural number.
For any action of $G$ on $X=\A^n_F$, one has
$H^i(X^G,\Z/\ell)=0$ for all $i>0$, and $X^G$ is smooth, connected
and non-empty. In particular, $X^G(F)$ is non-empty.
\end{cor}

Now let us look at some other cases where we can obtain a positive
answer to Question \ref{qu-serre}.

\subsection{The case where $n\leq 2$}\label{subsec-lowdim}
If $F$ is any field, the automorphism group of $\A^1_F=\Spec F[x]$
is easy to describe: a straightforward computation shows that it
is the group of affine transformations $$x\mapsto ax+b$$ with
$a\in F^{\times}$ and $b\in F$. Such an affine transformation $g$
has finite order $N>1$, with $N$ not divisible by the
characteristic of $F$, if and only if $a$ is a primitive $N$-th
root of unity.
 Then $x=b\cdot (1-a)^{-1}\in F$ is a fixed point, it is
the only one and it is rational.


\begin{prop}\label{prop-dim1}
Let $F$ be any field. If $G$ acts non-trivially on $\A^1_F$, then
there exists a unique fixed point, and it is $F$-rational.
\end{prop}
\begin{proof}
We argue by induction on $|G|$. If $|G|=\ell$ then $G$ is
generated by a single element $g$ which acts on $\A^1_F$ as an
affine transformation, so, as we have seen, there is a unique
fixed point, which is rational.

We assume that $|G|>\ell$, and  that we know the result for all
finite $\ell$-groups $H$ with $|H|<|G|$. As all finite
$\ell$-groups are solvable, $G$ admits a non-trivial normal
subgroup $H$. If $H$ acts trivially on $\A^1_F$, then we can apply
the induction hypothesis to the $G/H$-action on $\A^1_F$ and the
result follows. If $H$ acts non-trivially, then by the induction
hypothesis $(\A^1_F)^H$ is isomorphic to $\Spec F$, and the
$G/H$-action on $(\A^1_F)^H$ is trivial.
\end{proof}

\begin{remark}
We can also immediately deduce Proposition \ref{prop-dim1} from
Corollary \ref{cor-smith}, since that corollary implies that the
fixed locus $(\A^1_F)^G$ is smooth and geometrically connected
over $F$.
\end{remark}

Now we consider the case where $n=2$.

\begin{lemma}\label{lemm-indet}
Let $X, W$ be two normal irreducible varieties of dimension $2$
over a field $F$, and let $Y\supset X, \ Z\supset W$ be normal
compactifications. Let $g: X\to W$ be an isomorphism. If the
 irreducible components of $Z\setminus W$ of dimension one (with their reduced
induced structure) are geometrically  integral over $F$, then
every point of indeterminacy $y$ of $g$ on $Y$ is $F$-rational.
Moreover, the rational map $g^{-1}:Z\dashrightarrow Y$ contracts
an irreducible component of $Z\setminus W$ to $y$.
\end{lemma}
\begin{proof}
 Let $\Gamma_g\subset Y\times_F Z$ be the closure of the graph of
$g$, with its reduced induced structure. Let $p$, resp. $q$ be the
projection of $\Gamma_g$ to $Y$, resp. $Z$. Let $y$ be a closed
point of $Y$, with residue field $\kappa(y)$. If $p^{-1}(y)$ has
dimension zero, then by semi-continuity of fiber dimension
\cite[13.1.4]{ega4.3},
 there
exists an open neighbourhood $U$ of $y$ in $Y$ such that
$$p|_{p^{-1}(U)}: p^{-1}(U) \to U$$ is quasi-finite. Since $Y$ is normal
and $$p|_{p^{-1}(X)}: p^{-1}(X) \to X$$ is an isomorphism, we can
deduce from \cite[4.4.9]{ega3.1} that $p|_{p^{-1}(U)}$ is an
isomorphism, so that $g$ is defined at $y$.

Hence, if $y$ is a point of indeterminacy of $g$, then $p^{-1}(y)$
contains an irreducible component of dimension one. Its image
under $q$ is an irreducible closed subset $C$ of $Z$ of dimension
one. This subset $C$ is disjoint from $W$, because $q^{-1}(W)$ is
contained in $p^{-1}(X)$ (the graph of the isomorphism $g:X\to W$
is finite over $W$, and thus closed in $Y\times_F W$). It follows
that $C$ is an irreducible component of $Z\setminus W$. By our
hypotheses, if we endow $C$ with its reduced induced structure,
then $C$ is geometrically integral over $F$.

 Since $Z$ is normal, the rational map $g^{-1}$ is defined on a
dense open subscheme $C^0$ of $C$. By construction, $g^{-1}$
contracts $C^0$ to the point $y$. Since $C^0$ is geometrically
integral over $F$, the existence of a morphism $C^0\to \Spec
\kappa(y)$ implies that $\kappa(y)=F$ and $y$ is $F$-rational.
\end{proof}

\begin{lemma}\label{lemm-auto}
Let $F$ be any field, and consider an action of  $G$ on
$X=\A^2_F$.   Assume that the fixed locus $X^G$ has dimension one.
Then the following properties hold.

\begin{enumerate}
\item If $D$ is the regular completion of $X^G$, then $D\setminus
X^G$ consists of a unique point $x_{\infty}$, and the residue
field of $x_\infty$ is a purely inseparable extension of $F$.
Moreover, if $F^a$ is an algebraic closure of $F$, then the
normalization $D'$ of $D\times_F F^a$ is isomorphic to
$\Pro^1_{F^a}$.

\item Assume that $G$ is cyclic, and let $g$ be a generator. Let
$Y$ be a smooth compactification of $X$ such that the irreducible
components of $Y\setminus X$ are smooth curves that intersect each
other only at $F$-rational points. Denote by $C$ the closure of
$X^G$ in $Y$, and assume that the birational map
$$g:Y\dashrightarrow Y$$ is defined at the unique point
$c_{\infty}$ of $C\setminus X^G$. Then $X^G$ is isomorphic to
$\A^1_F$.
\end{enumerate}
\end{lemma}
\begin{proof}
(1) We denote by $F^a$ an algebraic closure of $F$. By Corollary
\ref{cor-smith}, we know that $X^G$ is smooth and geometrically
connected, and that
$$H^1(X^G\times_F F^a,\Z/\ell)=0.$$  Hence, $X^G$ is a smooth geometrically
connected affine
 curve and $\chi_c(X^G)=1$.

 The normalization morphism $n:D'\to D\times_F F^a$ is an isomorphism over the
smooth open subscheme $X^G\times_F F^a$ of $D\times_F F^a$. We have
 $$1<\chi_c(X^G)+\chi_c(D'\setminus n^{-1}(X^G\times_F F^a))=\chi_c(D')\leq 2.$$
 It follows that $D'$ is isomorphic to $\Pro^1_{F^a}$  and that
 $D'\setminus n^{-1}(X^G\times_F F^{a})$ consists of a unique point. This implies that
$D\setminus X^G$ consists of a unique
  closed point $x_{\infty}$, which
 is 
  purely inseparable over $F$.

(2) By (1), it is enough to show that $C$ is smooth over $F$ and
 that $c_{\infty}$ is separable over $F$: then $c_{\infty}$ is
 $F$-rational and $C\cong D$ is isomorphic to $\Pro^1_F$.

 For every $h\in G$, we denote by
$U_h$ the domain of definition of the birational map
$$h:Y\dashrightarrow Y$$  and
we set
$$U=\bigcap_{h\in G}U_h,\quad V=\bigcap_{h\in G}\{x\in U\,|\,h(x)\in U\}.$$
Then $V$ is an open subset of $Y$ that contains $X\cup
\{c_{\infty}\}$, and the $G$-action on $X$ extends uniquely to a
$G$-action on $V$, which stabilizes $V\setminus X$. By definition,
the fixed locus $V^G$ contains $C$. By Proposition
\ref{prop-smooth}, $V^G$ is smooth over $F$.  This implies that
$C$ is a connected component of $V^G$, which is smooth over $F$.

If $c_{\infty}$ lies on more than one irreducible component of
$Y\setminus X$, then by assumption, $c_{\infty}$ is $F$-rational.
If $c_{\infty}$ lies on precisely one irreducible component $E$ of
$Y\setminus X$, then $G$ acts on $V\cap E$, and $c_{\infty}$ is an
isolated fixed point of this action. Applying Proposition
\ref{prop-smooth} to the $G$-action on $V\cap E$, we see that
$c_\infty$ is smooth, thus separable over $F$.
\end{proof}

\begin{theorem}\label{thm-dim2}
Let $F$ be any field. For any $G$-action on $X=\A^2_F$, the fixed
locus $X^G$ is isomorphic to $\A^m_F$, with $m\in \{0,1,2\}$. In
particular, $X^G$ has an $F$-rational point.
\end{theorem}
\begin{proof}
By Proposition \ref{prop-dim1} and an induction argument as in the
proof of Proposition \ref{prop-dim1}, we may assume that
$|G|=\ell$.
 We may also assume that
$X^G$ has dimension one, since the case of dimension two is
obvious, and the case of dimension zero follows from Theorem
\ref{thm-smith}. We now prove the theorem under these assumptions.

Let $g$ be a generator of $G$. It defines an automorphism  of
$X=\A^2_F$ and, thus, a birational map
$$g: \mathbb{P}^2_F \dashrightarrow \mathbb{P}^2_F.$$ Denote by $C_0$ the
closure of $X^G$ in $\mathbb{P}^2_F$. By Lemma \ref{lemm-auto},
the boundary $C_0\setminus X^G$ consists of a unique point $c_0$.

If $g$ is defined at $c_0$, then Lemma \ref{lemm-auto} implies
that $X^G=\A^1_F$.  Hence, we may assume that $g$ is not
 defined at $c_0$. Then $c_0$ is $F$-rational by
Lemma \ref{lemm-indet}, applied to $W=X=\A^2_F$ and
$Y=Z=\Pro^2_F$.

Let $h_1:Y_1\to \mathbb{P}^2_F$ be the blow-up
 at $c_0$, and consider the birational map
$$g\circ h_1:Y_1\dashrightarrow \mathbb{P}^2_F.$$ We denote by 
 $C_1$ the
closure of $X^G$ in $Y_1$, and by $c_1$ the unique point of
$C_1\setminus X^G$. If $g\circ h_1$ is not defined at $c_1$, then
$c_1$ is $F$-rational by Lemma \ref{lemm-indet}, and we consider
the blow-up
$$h_2:Y_2\to Y_1$$
at $c_1$.

 If we continue in this way, then by resolution of indeterminacies \cite[9.2.7]{liu}, we will eventually construct a sequence
 $(Y_i,C_i,c_i)$ with $i=0,\ldots,N$ and morphisms $h_i:Y_i\to Y_{i-1}$ for $i=1,\ldots,N$ that satisfy the following
 properties:
\begin{itemize}
\item $Y_0=\mathbb{P}^2_F$, $Y_i$ is a smooth compactification of
$\A^2_F$ for every $i$, and $Y_i\setminus \A^2_F$ is a chain of
$(i+1)$ copies of $\mathbb{P}^1_F$ that intersect at $F$-rational
points, \item $C_i$ is the closure of $X^G$ in $Y_i$ and $c_i$ is
the unique point of $C_i\setminus X^G$,  \item the point $c_i$ is
$F$-rational for every $i<N$,
 \item $h_i:Y_i\to Y_{i-1}$ is the blow-up at $c_{i-1}$, for
 every $i\geq 1$,
 \item if we put $h=h_1\circ h_{2}\circ \ldots\circ h_{N}$, then the birational map $$g\circ h:Y_N\dashrightarrow \mathbb{P}^2_F$$
 is defined at $c_{N}$. We denote by $U$ the domain of definition
 of $g\circ h$.
\end{itemize}

For every $i$ in $\{0,\ldots,N\}$, the action of $g$ on $\A^2_F$
defines a birational map
$$g_i:Y_i\dashrightarrow Y_i.$$ We claim that, for some $i$,
either $g_i$ or $g_i^{-1}$ will be defined at $c_i$. This suffices
to prove the theorem, by Lemma \ref{lemm-auto}.

So let us prove the claim. Assume that $g_0^{-1}=g^{-1}$ is not
defined at $c_0$. Then, a fortiori, the rational map
$$(g\circ h)^{-1}:\mathbb{P}^2_F\dashrightarrow Y_N$$ is not
defined at $c_0$. By Lemma \ref{lemm-indet}, this means that
$g\circ h$ contracts an irreducible component of $Y_N\setminus
\A^2_F$ to $c_0$. By \cite[9.2.1]{liu}, the morphism $g\circ
h:U\to \mathbb{P}^2_F$ factors through the blow-up $h_1:Y_1\to
\mathbb{P}^2_F$ at $c_0$ (to be precise, in \cite[9.2.1]{liu} it
is assumed that $U$ is proper over $F$, but this is not necessary:
one can apply \cite[9.2.1]{liu} to a compactification of the
morphism $g\circ h:U\to \mathbb{P}^2_F$ with regular domain).

Proceeding inductively, we see that either $g_i^{-1}$ is defined
at $c_i$ for some $i>0$, or $g\circ h$ factors through a morphism
$g_N:U\to Y_N$ so that $g_N$ is defined at $c_N$.
\end{proof}
\begin{remark}
One can push the proof a bit further: if $g_i^{-1}$ is defined at
$c_i$ for some $i$, then the proof of Lemma \ref{lemm-auto} shows
that the whole $G$-action is defined on an open neighbourhood of
$C_i$ in $Y_i$. Since blowing up $G$-fixed closed points is
$G$-equivariant, we see that in all cases, the $G$-action is
defined on an open neighbourhood of $C_N$ in $Y_N$.
\end{remark}

 \begin{cor}\label{cor-dim2} Let $F$ be any field, and let
$H$ be a finite solvable group whose order is not divisible by the
characteristic of $F$. For any $H$-action on $X=\A^2_F$, the fixed
locus $X^H$ is isomorphic to $\A^m_F$, with $m\in \{0,1,2\}$. In
particular, $X^H$ has an $F$-rational point.
\end{cor}
\begin{proof}
This follows immediately from Proposition \ref{prop-dim1}, Theorem
\ref{thm-dim2} and the fact that $H$ has a composition series
whose composition factors are cyclic groups with order a prime
number different from the characteristic of $F$.
\end{proof}

\subsection{The case $F=\R$} \label{sub-r} Here we can use an argument
similar
to Section \ref{subsubsec-euler} to prove the following result.

\begin{prop}
Let $X$ be an $\R$-variety with good $G$-action. We consider
$X(\R)$ with its real topology. Then one has the following
congruence of singular Euler characteristics with compact
supports:
$$\chi_{c,\sing}(X(\R))\equiv \chi_{c,\sing}(X(\R)^G)\mod \ell.$$
In particular, if $\ell$ does not divide $\chi_{c,\sing}(X(\R))$,
then $X(\R)^G$ is non-empty.
\end{prop}
\begin{proof}
We can simply copy the proof of Proposition \ref{prop-euler},
using the fact that, if $|G|=\ell$ and $Y=X\setminus X^G$, the
projection
$$Y(\R)\to Y(\R)/G$$ is a topological covering of degree $\ell$.
\end{proof}

\begin{cor}
Any $G$-action on $\A^n_{\R}$, with $n$ a non-zero natural number, admits an
$\R$-rational fixed point.
\end{cor}

\subsection{The case where $F$ is a henselian discretely valued
field, of residual characteristic $\neq \ell$} \label{sub-hens}
Recall that the motivic Serre invariant $S(\cdot)$ has been
defined in Definitions \ref{def-LoSe} and \ref{def-char0} and that
the rational volume $s(\cdot)$ has been defined in Definition
\ref{def-ratl-vol}.

\begin{theorem}\label{thm-rat}Let $\ell$ be a prime number and $G$ a
finite $\ell$-group. Let $K$ be a henselian discretely valued
field whose residue field $k$ is perfect.  Let $X$ be a
$K$-variety with good $G$-action. Assume either that  $K$ has
characteristic zero, or that $X$ is smooth and proper over $K$.

\begin{enumerate}
\item Assume that the characteristic of $k$ is different from
$\ell$. Then one has
\begin{equation*}\label{eq-algcl}
s(X)\equiv s(X^G)\mod \ell.\end{equation*} In particular, if
$\ell$ does not divide $s(X)$, then $X^G(K^{sh})$ is non-empty.

\item Assume that $k$ is finite, of cardinality $q$. Then one has
\begin{equation*}\label{eq-finite}\sharp S(X)\equiv \sharp S(X^G)\mod
\gcd(\ell,q-1).\end{equation*}
 In
particular, if the residue class of $\sharp S(X)$ in
$\Z/(\ell,q-1)$ is non-zero
  (e.g. if $\ell$ divides $q-1$ and $\sharp S(X)=1$),
then $X^G(K)$ is non-empty.
\end{enumerate}
\end{theorem}
\begin{proof}

 First, assume that $K$
has characteristic zero. Then all the members of the congruences
in (1) and (2) are additive with respect to equivariant closed
immersions of $K$-varieties with $G$-action. More precisely,
taking the $G$-fixed locus defines a ring morphism
$$(\cdot)^G:\KG{K}\to \KVar{K}:[X]\mapsto [X^G].$$
Since the equivariant Grothendieck group $\KG{K}$ is generated by
the classes of smooth and proper $K$-varieties
\cite[7.1]{bittner}, it suffices to consider the case where $X$ is
smooth and proper.

So suppose that $K$ has characteristic $\geq 0$, that $\ell$ is
invertible in $k$, and that $X$ is smooth and proper over $K$. By
Proposition \ref{prop-weakG}, $X$ admits a weak N\'eron $G$-model
$\mathcal{X}$ over $R$. By Proposition \ref{prop-Gfix}, we know
that $\mathcal{X}^G$ is a weak N\'eron model for $X^G$. By
definition, one has $s(X)=\chi_c(\mathcal{X}_k)$ and
$s(X^G)=\chi_c(\mathcal{X}_k^G)$, so that (1) follows from
Proposition \ref{prop-euler}.

Likewise, if $k$ is finite, of cardinality $q$, one has by
definition $\sharp S(X)=|\mathcal{X}(k)|$ and $\sharp
S(X^G)=|\mathcal{X}^G(k)|$ in $\Z/(q-1)$. On the other hand, by
 Proposition \ref{prop-finfield}, one has
$|\mathcal{X}(k)| \equiv  |\mathcal{X}^G(k)| \ {\rm mod}  \ \ell$.
 Combined with Proposition \ref{prop-serrefinite}, this shows (2).
\end{proof}

\begin{cor}\label{cor-finite} Let the notations and assumptions
be as in Theorem \ref{thm-rat}. Assume in addition that $K$ has
characteristic zero. Let $n$ be a non-zero natural number.
\begin{enumerate}
\item If $k$ is algebraically closed and of characteristic $\neq
\ell$, then any $G$-action on $\A^n_K$  admits a $K$-rational
fixed point. \item If $k$ is finite, of cardinality $q$, and
$\ell$ divides $(q-1)$, then any $G$-action on $\A^n_K$ admits a
$K$-rational fixed point.
\end{enumerate}
\end{cor}


\begin{question}
If $k$ is finite, can we get rid of the condition that $\ell$
divides $(q-1)$ in Corollary \ref{cor-finite}?
\end{question}

Even if $R$ has mixed characteristic, Theorem \ref{thm-rat} is
false if we do not assume that $\ell$ is invertible in $k$, as is
shown by the following example.

\begin{example}
Let $k$ be an algebraic closure of $\mathbb{F}_3$, let $R=W(k)$,
$K=\mathrm{Frac}(R)$ and $G=\Z/3$. Denote by $g$ the class of $1$
in $G$. Consider the action of $G$ on
$$X=\mathrm{Proj}\,K[x,y]$$ defined by
$$g(x:y)=(y:-x-y).$$
Let $K^a$ be an algebraic closure of $K$. The points of $X^G(K^a)$
are the points with homogeneous coordinates $(1:\omega)$ with
$\omega^2+\omega+1=0$. Since $\mu_3(K)=\{1\}$ we find that
$X^G(K)$ is empty. It follows that $S(X^G)=0$, while $S(X)=2$ in
$\KR{k}/(\LL_k-1)$.

Note that the $G$-action on $X$ extends uniquely to an action of
$G$ on the weak N\'eron model
$$\mathcal{X}=\mathrm{Proj}\,R[x,y]$$ of $X$. The $G$-action on
$\mathcal{X}_k$ has a fixed point (namely $(1:1)$) but this fixed
point does not lift to an element of $\mathcal{X}^G(R)=X^G(K)$.
Such a situation cannot occur when $\ell$ is invertible in $k$, by
smoothness of $\mathcal{X}^G$ over $R$.
\end{example}

\section{The trace formula}\label{sec-trace}
 Let $K$, $R$
and $k$ be as in Section \ref{sec-notation}.
 If $k$ is
algebraically closed and of characteristic zero, the rational
volume $s(X)$ of a $K$-variety $X$ admits a cohomological
interpretation in terms of a trace formula, which is similar to
the Grothendieck-Lefschetz-Verdier trace formula that we've
discussed in Section \ref{verdier}.

\begin{theorem}\label{thm-trace} Let $K$ be a strictly henselian
discretely valued field with residue field $k$, and assume that
$k$ has characteristic zero. Let $K^a$ be an algebraic closure of
$K$, and choose a topological generator $\varphi$ of the absolute
Galois group $G(K^a/K)$ of $K$. Then for every $K$-variety $X$, one
has
$$s(X)=\sum_{i\geq 0}(-1)^i
\mathrm{Trace}(\varphi\,|\,H^i_c(X\times_K K^a,\Q_\ell))$$ for
every prime $\ell$.
\end{theorem}
\begin{proof}
If $K$ is complete, this is \cite[6.5]{nicaise}. The proof remains
valid in the general case. Alternatively, one can observe that
both terms of the trace formula are invariant under completion
of $K$, by Lemma \ref{lemm-complete}, invariance of $\ell$-adic
cohomology under algebraically closed field extensions, and the
fact that $G(\widehat{K}^a/\widehat{K})\to G(K^a/K)$ is an
isomorphism.
\end{proof}

There are partial results in this direction in arbitrary characteristic; see
\cite[\S\,4]{nicaise-tame}.

\begin{prop}\label{prop-acyc}
We keep the assumptions of Theorem \ref{thm-trace}. Let $\ell$ be
a prime, and $G$ a finite $\ell$-group. If $X$ is a smooth
 geometrically connected $K$-variety with good $G$-action such that
$$H^i(X\times_K K^a,\Z/\ell)=0$$ for all $i>0$, then
$$s(X)=s(X^G)=1.$$ In particular, $X^G$ has a
$K$-rational point.
\end{prop}
\begin{proof}
By Theorem \ref{thm-smith}, we know that
$$H^i(X^G\times_K K^a,\Z/\ell)=0$$ for all $i>0$, and that $X^G$ is geometrically connected.
 It follows that
$$H^i(X^G\times_K K^a,\Q_\ell)=0$$ for all $i>0$.
 The fixed locus $X^G$ is smooth over $K$ by Proposition
\ref{prop-smooth}, so that we can apply Poincar\'e duality, and we
obtain that
$$H^i_c(X^G\times_K K^a,\Q_\ell)=0$$ for all $i\neq 2d$, where $d$ denotes the dimension of $X^G$.
 Since $G(K^a/K)$ acts trivially on
 $$H^{2d}_c(X^G\times_K K^a,\Q_\ell)\cong \Q_\ell(-d),$$ Theorem
 \ref{thm-trace} implies that $s(X^G)=1$. The same argument, where one assumes the $G$-action on $X$ to be trivial, shows
 that $s(X)=1$.
\end{proof}

\begin{cor}
We keep the assumptions of Theorem \ref{thm-trace}.  Let $n$ be a
non-zero natural number. Let $\ell$ be a prime, and $G$ a finite
$\ell$-group.
 For any action of $G$ on $X=\A^n_K$, one has
$$s(X^G)=1.$$
\end{cor}
\begin{question} Does this Corollary also hold if $K$ has mixed
characteristic and $\ell$ is invertible in $k$? Is it true that
one has the stronger equality
$$S(X^G)=1$$ in $K_0(\Var_k)/(\LL_k-1)$ if $k$ has characteristic zero?
\end{question}

\section{Relation with conjectures in affine geometry}
It is clear that Serre's question (\ref{qu-serre}) is closely related to the
study of the structure of the automorphism group of affine spaces.
If the base field $F$ is the field $\C$ of complex numbers, this
study is an important part of complex affine geometry. The
following question is mentioned in \cite{kraft}.

\begin{question}[Linearization Problem]\label{qu-lin}
Let $n$ be a non-zero natural number, and let $g$ be an
automorphism of $\A^n_{\C}$ of finite order. Can we always find a
coordinate system on $\A^n_{\C}$ such that $g$ becomes linear? In
other words, is $g$ conjugate to an element of $GL_{n}(\C)$ in the
automorphism group $\Aut_{n,\C}$ of $\A^n_{\C}$?
\end{question}

The question has been answered affirmatively if $n<3$, but it is
open for $n\geq 3$. In the cases $n=1$ and $n=2$, the automorphism
group $\Aut_{n,\C}$ is well-understood, but its structure is much
more complicated when $n\geq 3$. We've seen in Section
\ref{subsec-lowdim} that, for every field $F$, the automorphisms
of $\A^1_F$ are precisely the invertible affine transformations.
In dimension $n=2$, the theorem of Jung-van der Kulk \cite{vdKulk}
states that the automorphism group $\Aut_{2,F}$ of $\A^2_F$ is the
amalgamated product of the subgroup $\Aff_{2,F}$ of invertible
affine transformations and the subgroup $\mathscr{J}_{2,F}$ of
so-called triangular automorphisms over their intersection
$\mathscr{B}_{2,F}=\Aff_{2,F}\cap \mathscr{J}_{2,F}$. The group
$\mathscr{J}_{2,F}$ is also called the Jonqui\`ere subgroup. For a
geometric proof and some historical background, we refer to
\cite{lamy}. The theorem of Jung-van der Kulk implies by an
elementary induction argument on the length of an element of the
amalgamated product
$$\Aff_{2,F}\ast_{\mathscr{B}_{2,F}}\mathscr{J}_{2,F}$$
that every automorphism $g$ of $\A^2_F$ of finite order $N$ is
conjugate  in $\Aut_{2,F}$ to an element of $\Aff_{2,F}$ or an
element of $\mathscr{J}_{2,F}$ (see the discussion of Theorem 7 in
\cite{kraft}). If $N$ is not divisible by the characteristic of
$F$, then \cite[2.1]{ivanenko} implies that $g$ is conjugate to an
element of $\Aff_{2,F}$. Since we know by Corollary
\ref{cor-euler} that the automorphism $g$ admits a fixed point
over a separable closure of $F$, it follows that $g$ is
linearizable, i.e., conjugate to an element of $GL_2(F)$. This
yields an alternative proof of Theorem \ref{thm-dim2}.


 In
dimension $3$, we get the following conditional result.

\begin{prop}
Assume that the Linearization Problem has a positive solution for
$n=3$. Let $F$ be a field of characteristic zero, and let $G$ be a
finite solvable group. For any action of $G$ on $X=\A^3_F$, the
fixed locus $X^G$ is isomorphic to $\A^m_F$ with $m\in
\{0,1,2,3\}$. In particular, $X^G(F)$ is non-empty.
\end{prop}
\begin{proof}
We may assume that $F$ is a subfield of $\C$, and that $G$ acts
non-trivially on $X$. By Proposition \ref{prop-dim1} and Theorem
\ref{thm-dim2} and an induction argument as in the proof of
Proposition \ref{prop-dim1}, we may suppose that $G$ is cyclic.
 Let $g$ be a generator.  By our assumption in the statement of the
proposition, the action of $g$ on $\A^3_{\C}$ can be linearized by
a suitable choice of coordinates. This implies that $X^G\times_F
\C$ is isomorphic to $\A^m_{\C}$, with $m\in \{0,1,2\}$. But
$\A^m_{\C}$ has no non-trivial forms over $F$. For $m=0$ this is
obvious. For $m=1$ it can be shown as in the beginning of the
proof of Lemma \ref{lemm-auto}(1):  if $C^0$ is a form of
$\mathbb{A}^1$ over $F$, and $C$ is a smooth compactification,
then singular cohomology of $C$ over  $\C$ implies that
$C\setminus C^0$ is geometrically irreducible, thus is a rational
point.  For $m=2$ it is proven in \cite{shafarevich} and
\cite{kambayshi}.
\end{proof}

\begin{question}\label{qu-fixaffine}
Let $G$ be a finite solvable group, and let $F$ be a field of
characteristic zero. Is it true that, for every integer $n>0$ and
every $G$-action on $X=\A^n_F$, the fixed locus $X^G$ is
isomorphic to $\A^m_F$ for some $m\in \{0,\ldots, n\}$?
\end{question}

 This would be compatible with
the results in Sections \ref{sec-serre} and \ref{sec-trace}, in
particular Corollary \ref{cor-smith}. Note that, in order to show
that Question \ref{qu-fixaffine} has a positive answer, it
suffices to deal with the case where $G$ is a cyclic group of
prime order, by induction on the order of $G$.

 It seems that, in order
to
 answer Question \ref{qu-serre} in full generality, the
  tools of the previous sections do not suffice, and that one
  should exploit the specific geometric structure of the affine
  space.

\end{document}